\documentclass[12pt]{amsart}    

\usepackage{amsmath,amssymb,amsthm,enumerate,mathabx}

\newtheorem{Theorem}{Theorem}[section]

\newtheorem{Lemma}[Theorem]{Lemma}
\newtheorem{Corollary}[Theorem]{Corollary}
\newtheorem{Theorem1}{Theorem}
 
\theoremstyle{definition}
\newtheorem{Definition}[Theorem]{Definition}
\newtheorem{Fact}[Theorem]{Fact}

\newtheorem{Remark}[Theorem]{Remark}
\newtheorem{Question}{Question}

\DeclareMathOperator{\cl}{cl_S}
\DeclareMathOperator{\dcl}{dcl}
\DeclareMathOperator{\tp}{tp}

\DeclareMathOperator{\dom}{dom}
\DeclareMathOperator{\acc}{acc}

\usepackage[margin=23mm]{geometry}

\title{Vaught's conjecture for theories of discretely ordered structures}
\author{Predrag Tanovi\'c}
\thanks{This research was supported by the Science Fund of the Republic of 
Serbia, Grant No. 7750027: Set-theoretic, model-theoretic and 
Ramsey-theoretic phenomena in mathematical structures: similarity and 
diversity–SMART}
\address{Mathematical Institute SANU, Belgrade, Serbia} 
\email{tane [at] mi.sanu.ac.rs}
\begin{document}
\maketitle
\begin{abstract} 
Let $T$ be a countable
complete first-order theory with a definable, infinite, discrete linear order. We prove that $T$ has continuum-many countable models. The proof  is purely first-order,  but raises the question of Borel completeness of $T$.
\end{abstract}

\bigskip
Throughout the paper, $L$ will be an at most countable  language and $T$  an arbitrary (possibly multi-sorted), complete,  first-order theory with infinite models;  $I(\aleph_0,T)$ will denote the number of countable models of $T$, up to 
isomorphism.  We say that $T$  admits an infinite,  discrete, linear order, if such a parametrically definable order $(D,<)$ can be found in some   $\aleph_0$-saturated model $M\models T$ with $D\subseteq M^n$; by adding an extra sort, if necessary, we may always assume $D\subseteq M$. 
Our goal in this paper is to present a complete, self-contained, elementary   proof of the following theorem.

\begin{Theorem1}
If $T$ admits an infinite, definable, discrete, linear order, then $I(\aleph_0,T)=2^{\aleph_0}$. 
\end{Theorem1}

We will also prove that the conclusion of Theorem 1 holds for theories admitting a relatively definable discrete order on the locus of some complete type; that  will be useful in the continuation of our work, started in \cite{MT}, on Vaught's conjecture for weakly quasi-o-minimal theories. 

The original Vaught's conjecture, see \cite{Vaught}, states that  a complete, countable theory $T$ either has at most countably many or continuum-many non-isomorphic countable models, independently of the continuum hypothesis. Concerning theories of ordered structures, Rubin in \cite{Rubin} confirmed the conjecture for theories of colored orders (linear orders with unary predicates), Shelah in \cite{Sh} for theories with a definable linear order and Skolem functions, Mayer in \cite{Mayer} for $o$-minimal theories  and Moconja and the author in \cite{MT} for (a  superclass of) binary, weakly quasi-$o$-minimal theories. 
In \cite{Steel}, Steel used descriptive set theory and infinitary model theory methods to prove a stronger version of the conjecture for theories of trees: every consistent $L_{\omega_1\,\omega}$-sentence has either countably many or perfectly many countable models. It has been known, but up to our knowledge never published, that Steel's methods might lead to a proof of Theorem 1. 
However, in this paper we will show that there are (at least) three sources from which an elementary, first-order proof of Theorem 1 may be deduced: Shelah's paper \cite{Sh}, the author's paper on types directed by constants \cite{T}, and the paper \cite{DI} by Ili\'c.
Since our aim here is to present a complete, but as short as possible proof, the third source turned out to be the most convenient. 
The central notion in \cite{DI} is that of a  {\it simple} type in discretely ordered structures, see Definition \ref{Definition_simple} below. It is  rather routine to show if the thory $T$ is small and interprets an infinite definable\footnote{Throughout the paper, definable means without  parameters.} discrete order $(D,<)$, then simple types exist in $T(c)$ for all $c\in D$; that  will be done in Lemma \ref{Lemma_simpleExist} below. Then we will show how Theorem 1 can be deduced from either \cite{Sh} or \cite{T}. But, in order to present a self-contained proof of Theorem 1, 
in Section 2 we will modify arguments from \cite{DI} and prove the following theorem.

\begin{Theorem1}
Suppose that $p\in S_1(T)$ is a simple type, as witnessed by $(C,<)$ and $(D,<)$. Then for every countable, discrete linear order without end-points $\mathbb L$, there is a countable model $M_{\mathbb L}\models T$ such that $(p(M_{\mathbb L}),<)\cong \mathbb L$. In particular,   $I(\aleph_0,T)=2^{\aleph_0}$. 
\end{Theorem1}

Theorem 1 is a consequence of Theorem 2: If $T$ is not small, then  $I(\aleph_0,T)=2^{\aleph_0}$ follows, so let $T$ be a small theory with an infinite discrete  order $(D,<)$ definable in a model of $T$. Then, by Lemma \ref{Lemma_simpleExist}, there exists a simple type in $S_1(c)$ for all $c\in D$, so $I(\aleph_0,T(c))=2^{\aleph_0}$ follows by Theorem 2. Since in any countable model of $T$ a new constant can be interpreted in countably many ways,  $I(\aleph_0,T)\cdot\aleph_0\geq  I(\aleph_0,T(a)) $ follows.
Therefore, $I(\aleph_0,T)=2^{\aleph_0}$ holds.  

\smallskip
The conclusion of Theorem 2, that any countable linear order can be ``coded'' in a model of $T$, is stronger than $I(\aleph_0,T)=2^{\aleph_0}$ and 
naturally raises the following question.

\begin{Question}
Is every  countable, complete theory that interprets an infinite discrete order Borel complete?
\end{Question}

Section 1 contains preliminaries, basic properties of simple types and sketches of proofs of Theorem 1 that rely on \cite{Sh} and \cite{T}. In Section 2 arguments from \cite{DI} are adapted to prove Theorem 2. Section 3 contains corollaries of Theorem 1, further  discussion on Question 1 and further questions.

\section{Preliminaries}

Let $(P,<)$ be a linear order and $X,Y\subseteq P$. $X<Y$ means that $x<y$ for all $x\in X$ and $y\in Y$.
$x<Y$ means  $\{x\}<Y$.  By $[x,y]$ we will denote the closed interval $\{z\in P\mid x\leq z\leq y\}$.  $X$ is an {\it initial part} if $x\in X$ and $y<x$ imply $y\in X$; $X$ is {\it convex} if $x,y\in X$ and $x<y$ imply $[x,y]\subseteq X$.
The element $y\in P$ is an immediate successor of $x\in P$, denoted by $S(x)=y$, if $x<y$ and no element of $P$ is strictly between $x$ and $y$. We will view $S$ as a partial function on $P$ and $S^m$ will be the $m$-th iterate of $S$; since $S$ is one-to-one, $S^m$ is a well-defined partial function for all  integers $m$. The minimal element of $X$ (or the left end of $X$), if it exists, is denoted by $\min X$; similarly for $\max X$. $(P,<)$ is a discrete order if the induced order-topology is discrete; equivalently: $S(x)$ is defined for all elements of $P$ except $\max P$ (if it exists) and  $S^{-1}(x)$ is defined for all elements of $P$ except $\min P$ (if it exists). 

Our model theoretic notation is mainly standard. 
We work within a fixed  monster model $M\models T$.  $a,b,c,a'...$ denote elements,  $\bar a,\bar b,...$ tuples of elements, and $A,B,...$ small (usually countable)  subsets of $M$; the letters $D,P,X,...$ may denote  any sets. Definability means without parameters and by partial types we will mean those that are closed under finite conjunctions. If $D\subseteq M$ and $\theta(\bar x)$ is a $L_M$-formula,  then  
 $\theta(D)=\{\bar d\in D^{|\bar x|}\mid M\models \theta(\bar d)\}$ is the set of realizations of  $\theta(\bar x)$ in $D$; similarly for partial  types. $\dcl(A)$ denotes the definable closure of $A$ in $M$: $a\in\dcl(A)$ means that there is a $L_A$-formula whose only realization in $M$ is $a$.  
$S_n(A)$ denotes the space of all complete $n$-types with parameters from $A$ endowed with the formula topology (basic clopen sets are $U_{\phi}=\{q\mid \phi(\bar x)\in q\}$ $\phi(\bar x)\in L_A$); 
$S_n(T):=S_n(\emptyset)$). The theory $T$ is {\it small} if   $\bigcup_{n\in\mathbb N}S_n(T)$ is countable. If $T$ is not small,  then  $I(\aleph_0,T)=2^{\aleph_0}$. 

Let $\Pi(\bar x)$ be a partial type over $A$ and $X\subseteq \Pi(M)$. We say that the set $X$ is relatively $A$-definable within $\Pi(M)$ if there is a $L_A$-formula $\theta(x)$ with $X=\Pi(M)\cap \theta(M)$; in this case we say that $\theta(x)$ relatively defines $X$ within $\Pi(M)$. 
If the sets $X_1,X_2$ are type-definable over $A$, then the function $f: X_1\to X_2$  is relatively $A$-definable if its graph is such.

\begin{Fact}\label{Fact1}
(a) Any relatively definable linear order has a definable extension. More precisely, if $\Pi(x)$ is a partial type over $A$ and $\phi(x,y)$  an  $L_A$-formula which relatively defines a linear order on $\Pi(M)$, then $\phi(x,y)$ defines a linear order on $\theta(M)$ for some  formula $\theta(x)\in \Pi(x)$.

(b)  Any relatively definable function has a definable extension. 
\end{Fact}
\begin{proof}
Routine compactness, we sketch only proof of part (a). Define:
$$\lambda(x,y,z):=\lnot\phi(x,x)\land (x\neq y\rightarrow \lnot(\phi(x,y)\leftrightarrow\phi(y,x)))\land (\phi(x,y)\land\phi(y,z)\rightarrow\phi(x,z)).$$
Since the formula $\phi(x,y)$ defines a linear order on $\Pi(M)$ we have:
$$\Pi(x)\cup\Pi(y)\cup\Pi(z)\vdash \lambda(x,y,z) .$$
By compactness there exists a formula $\theta(x)\in \Pi(x)$ with  $ \{\theta(x),\theta(y),\theta(z)\}\vdash\lambda(x,y,z)$.  
\end{proof}

Let  $\pi(x)$ be a partial type over $A$  and  $C\subseteq \dcl(A)$. We say that $\pi(x)$ is {\it finitely satisfied} in $C$ if every  formula of $\pi(x)$ is satisfied by some element of $C$. A complete type $p\in S_1(A)$ is finitely satisfied in $C$ if and only if $p$ is in the topological closure of $\mathcal C=\{\tp(c/A)\mid c\in C\}$; here, either  $p\in\mathcal C$, or $p$ is an accumulation point of $\mathcal C$; in the latter case,   we say that $p$ is a $C$-type. If the set $A$ is countable, then the space $S_1(A)$ is compact and separable, so whenever $C$ is infinite then a $C$-type $p\in S_1(A)$ exists. $p$ is the unique $C$-type in $S_1(A)$ if and only if for all $L_A$-formulae $\phi(x)$, exactly one of the sets $\phi(C)$ and $\lnot\phi(C)$ is infinite.

\begin{Definition}\label{Definition_simple}
A complete type $p\in S_1(A)$ is {\it simple} if there is an infinite set $C\subset\dcl(A)$ and an $A$-definable linear order $(D,<)$ such that the following two conditions hold:

(1) $p$ is the unique $C$-type in $S_1(A)$;

(2) $C$ is an initial part of $D$ ordered in the order-type $\boldsymbol{\omega}$.
\end{Definition}

\begin{Remark}
(i) Condition (1) in the previous definition may be replaced by the following, equivalent one \ : 
  for all $\phi(x)\in p$ the set $\lnot \phi(C)$ is finite. 

\smallskip
(ii)  If $p\in S_1(T)$ is a simple type, as witnessed by $C$ and $(D,<)$, then $p(M)\subset D$ holds: Indeed, any formula defining $D$, say $x\in D$, is satisfied by all elements of $C$, so, since $p$ is a $C$-type, $(x\in D)\in p$ and hence $p(M)\subseteq D$. 

\smallskip (iii) Observe that we did  not require in the definition that the witnessing order $(D,<)$ is discrete.  However, by Lemma \ref{Lemma_basic_simple}(b)  a discrete witness will always exists. 
\end{Remark}

\begin{Lemma}\label{Lemma_basic_simple}
If   $p\in S_1(A)$ is a simple type, as witnessed by $C\subset\dcl(A)$ and $(D,<)$, then:

(a) $P(M):=C\cup p(M)$ is an initial part  of $(D,<)$ that is type-definable over $A$.

(b) There exists an  $A$-definable, discretely ordered, initial part $D_0\subseteq D$  with $P(M)\subseteq D_0$. In particular, $(p(M),<)$ is a discrete order.  

(c) Every nonempty, relatively $M$-definable subset of  $P(M)$ has a  minimum.  

(d) If $C$ and $(D_1,<_1)$ also witness the simplicity of $p$ such that $<_1$ and $<$ agree on $C$, then they agree on $P(M)$, too. 
\end{Lemma}
\begin{proof}To simplify the notation, assume $A=\emptyset$ and  write $x<y$ instead of $x\in D\land y\in D\land x<y$. 

(a) Since $C$ is an initial part of $(D,<)$,    $C<p(M)$ holds. 
To prove that $P(M)$ is an initial part of $D$, it suffices to show that $a\in p(M)$ and $b< a$ together imply $b\in P(M)$. Let $\phi(x)\in \tp(b)$. 
Then  $\exists y(\phi(y)\land y< x)\in\tp(a)=p$. Since $p$ is   a $C$-type, there are $c\in C$ and $b'\in D$ with $\models\phi(b')\land b'< c$.
Then $b'< c$ implies $b'\in C$, so   $\phi(x)$ is realized by $b'\in C$. We have just shown that every formula $\phi(x)\in \tp(b)$ is realized in $C$, so the type $\tp(b)$ is finitely satisfied in $C$; 
If $\tp(b)$ is algebraic  then $b\in C$; otherwise, $\tp (b)$ is a $C$-type, so $\tp(b)=p$ holds by the uniqueness of $p$.
In both cases we have $b\in P(M)$, so $P(M)$ is an initial part of $D$. Then it is is easy to see that $P(M)$ is defined by  $P(x):=\{\exists y(\theta(y) \land   x<y)\mid \theta(y)\in p(y)\}$. 

\smallskip
(b) Let $c_0=\min C$ and let $\phi(x)$ be a formula expressing that the interval $[c_0,x]$ is discretely ordered by $<$.  For all $c\in C$ the interval $[c_0,c]$ is finite, so $\models  \phi(c)$ holds. Since $p$ is a $C$-type, $\phi(x)\in p$. Let $D_0$ be the union of all intervals $[c_0,a]\subset D$ where $a\in\phi(M)$.  

\smallskip
(c) Let $\emptyset\neq D_1\subset P(M)$  where $D_1$ is relatively defined by $\theta(x,\bar b)$ ($\bar b\in M$). Pick $a\in p(M)$ with $[c_0,a]\cap D_1\neq\emptyset$.  Let $\phi(x,\bar y)$ be a formula expressing that the set
$[c_0,x]\cap\theta(M,\bar y)$ is nonempty and has minimum. 
In particular, $\models\phi(a,\bar b)$ holds. For all $c\in C$ the interval $[c_0,c]$ is finite, so $\models \forall \bar y \left(\exists z(c_0\leq z\leq c\land \theta(z,\bar y))\rightarrow \phi(c,\bar y)\right)$. Since $p$ is a $C$-type and $a\in p(M)$, 
\begin{center} 
$\models \forall \bar y \left(\exists z(c_0\leq z\leq a\land \theta(z,\bar y))\rightarrow \phi(a,\bar y)\right)$.  
\end{center}
In particular,   $\models  \exists z(c_0\leq z\leq a\land \theta(z,\bar b))\rightarrow \phi(a,\bar b)$.  Now  $[c_0,a]\cap D_1\neq \emptyset$ implies $\models  \exists z(c_0\leq z\leq a\land \theta(z,\bar b))$, so $\models\phi(a,\bar b)$ and $\min ([c_0,a]\cap D_1)$ exists. Hence $\min D_1$ exists, too.

\smallskip
(d) Suppose not. Then for all $a\in p(M)$ the set $D(a)$ defined by $x<a\land x>_1 a$ is nonempty so, by part (c), $b=\min_{<} D(a)$ is well-defined. By part (a), $C<_1p(M)$ implies $b\notin C$, so  $b<a$ implies $b\in p(M)$. In particular, $D(b)\neq\emptyset$ and for any $b'\in D(b)$ we have $b'<b<a\land b'>_1b>_1 a$, contradicting the minimality of $b$ in $D(a)$. 
\end{proof}

In Section 2 we will work  within the order $(P(M),<)$, which is an initial part of any  order $(D,<)$ that witnesses the simplicity of $p$. By Lemma \ref{Lemma_basic_simple}(d) the order $(P(M),<)$ is determined by   $(C,<)$, so we will simply say that the order $(C,<)$ witnesses the simplicity of the type $p$; $(D,<)$ will be any discrete order, whose initial part is $C$. Note that the  successor function  of $(P(M),<)$ agrees  with the restricted  successor function of $(D,<)$.

\begin{Lemma}\label{Lemma_simpleExist}
Assume $|S_2(T)|\leq\aleph_0$ and  let $(D,<)$ be an infinite discrete order definable in $M$. Then for all $c\in D$ there exists a simple type in $S_1(c)$.
In particular, if $I(\aleph_0,T)<2^{\aleph_0}$ and $T$ interprets an infinite  discrete order, then a simple type exists in $T(c)$.
\end{Lemma}
\begin{proof}
By reversing the order, if necessary, we may assume that the set $C_0=\{S^n(c)\mid n\in \omega\}$ is infinite.  Let $\mathcal C=\acc \{\tp(S^n(c)/c)\mid n\in\omega \}$
be the set of all $C_0$-types in $S_1(c)$. Since $|S_2(T)|\leq\aleph_0$, $\mathcal C$ is  countable so, being closed, contains an  isolated point $p\in \mathcal C$. Choose a  
$\phi(x)\in p(x)$ isolating  $p$ within $\mathcal C$ and let $C=\phi(C_0)$; then $p$ is a unique $C$-type in $S_1(c)$.   If $\psi(x):=\phi(x)\land x\in D \land c\leq x$, then $C$ and $(\psi(D),<)$  witness that $p$ is a simple type.
\end{proof}

\noindent
{\it First proof of Theorem 1.} \ The notion of a type directed by constants is introduced in \cite{T}.  A type $p\in S_1(T)$ is $(C,\leq)$-directed  if:
 
 (1)  $\leq$ is a definable partial order on $M$; 
 
 (2)  $C\subseteq \dcl(0)$  is an initial part of $(M,\leq)$;

  (3) $p=\{\phi(x)\mid \mbox{$\phi(C)$ is co-finite in $C$}\} $.\\
If $T$ is small and interprets an infinite discrete order, then by Lemma \ref{Lemma_simpleExist}, there exists  a simple type $p\in S_1(T(c))$. Clearly, $p$ is directed  by constants (in $T(c)$), so by Corollary 1 of \cite{T}  $I(\aleph_0,T(c))=2^{\aleph_0}$ holds; so does \  $I(\aleph_0,T)=2^{\aleph_0}$.  \qed

\medskip
Before sketching the second proof of Theorem 1, let us make precise  what a relativization   $T_{\theta}$ of $T$ is.
Choose a language $L^*$ containing for each $L$-formula $\phi(\bar x)$ an $|\bar x|$-ary relation symbol $R_{\phi}$. In our model $M$, the predicate $R_{\phi}$ is interpreted naturally as $\phi(M)$. Let $M^*$ be the induced  $L^*$-structure and $T^*$ its theory. Note that $M$ and $M^*$ are interdefinable (have the same  definable sets). Moreover, any model of $T$ is interdefinable with  a model of $T^*$ and vice versa.  In particular, $I(\aleph_0,T)=I(\aleph_0,T^*)$. 
Let $\theta(x)$  be an $L$-formula, let $\theta(M)^* $ be a $L^*$-substructure of $M^*$ and let $T_{\theta}$ be its complete $L^*$-theory. It is a rather straightforward application of the omitting types theorem, noticed by Vaught, that for all countable  $N\models T_{\theta}$ there is a countable $M_N\models T$ with $\theta(M_N)^*\cong N$. Hence $I(\aleph_0,T_{\theta})=2^{\aleph_0}$ implies $I(\aleph_0,T)=2^{\aleph_0}$.

\medskip 
\noindent{\it Second proof of Theorem 1.} 
Suppose that $T$ is small and   interprets an infinite discrete order. Let $p\in S_1(c)$ be a simple type, as witnessed by $(C,<)$. Put $c_0=\min C$, fix $a\in p(M)$ and let $D_0=[c_0,a]$.
Then $D_0$ is an $ac$-definable set, by $\theta(x)$ say. By Lemma \ref{Lemma_basic_simple}(c) every $L_M$-definable subset of $D_0$ has minimum, so the relativization $T_{\theta}$ has definable   Skolem functions $f_{\phi(x,\bar y)}(\bar y)=\min\phi(M_{\theta},\bar y)$ (defined on $\exists x\phi (x,\bar y) (M)$). By Shelah's Theorem 3.8 from \cite{Sh}  $I(\aleph_0,T_{\theta})=2^{\aleph_0}$ holds. Hence $I(\aleph_0,T)=2^{\aleph_0}$ holds, too.\qed

\section{Proof of Theorem 2}

In this section we will simplify arguments from Section 2 of \cite{DI} and use them to prove Theorem 2. The crux of the proof is in the next lemma (Lemma 2.6 in \cite{DI}).

\begin{Lemma}\label{Lemma_f_is_Sn}
Suppose that  $p\in S_1(T)$ is a simple type, as witnessed by $(C,<)$ and $(D,<)$.  If $f:p(M)\to p(M)$ is a relatively definable function, then $f=S^k$ for some $k\in\mathbb Z$, where $S$ is the successor function on $(p(M),<)$.
\end{Lemma}
\begin{proof}Let $c_0=\min C$, $c_n=S^n(c_0)$ for $n\in\mathbb N$, and  suppose that $(D,<)$ is discrete.   
First, we will prove that $f$ is an automorphism of $(p(M),<)$. Choose a definable function $\hat f:D_1\to D_2$ extending $f$. Since $p(x)\cup\{\hat f(x)=a\}$ is consistent for some (equivalently all) $a\in p(M)$ and $M$ is $\aleph_0$-saturated, $f$ is surjective.
To verify the injectivity of $f$, let us consider the set $\hat f^{-1}(\{f(a)\})\cap [c_0,a]\subset p(M)$. It is an $a$-definable subset of $p(M)$ so, by Lemma \ref{Lemma_basic_simple}(c), has minimum, say $a'$; clearly,   $a'=\min(\hat f^{-1}(\{f(a')\})\cap [c_0,a'])$. Since  $a'\in p(M)$, the formula  $x=\min(\hat f^{-1}(\{f(x)\})\cap [c_0,x])$ belongs to $p(x)$; it follows that $f$ is injective. 
In particular, the inverse function $f^{-1}$ is relatively definable, so after possibly shrinking $D_1$ and $D_2$, we will assume that $\hat f$ is a bijection.
It remains to show that $f$ (equivalently $f^{-1}$) is order-preserving. 
Note that at least one of the formulae $\hat f(x)\leq x$ and $\hat f^{-1}(x)\leq x$ belongs to $p(x)$. We will from now on assume $(\hat f(x)\leq x)\in p(x)$, the proof in the other case is similar.
Suppose, by way of contradiction, that $f$ is not order-preserving. Then the set $D(a)$ defined by $\hat f(a)\leq \hat f(x)<x<a$ is nonempty (for all $a\in p(M)$). Observe that  $D(a)\subset [f(a),a]\subset p(M)$,  so  $\emptyset\neq D(b)\subset p(M)$ holds for all $b\in D(a)$. Choose $b'\in D(b)$. Then:  
$$\hat f(a)\leq \hat f(b)\leq \hat f(b')<b'<b<a$$ witnesses non-minimality of $b$ in $D(a)$. Hence no element of $D(a)$ is minimal, contradicting Lemma \ref{Lemma_basic_simple}(c). Therefore, $f$ is an automorphism of $(p(M),<)$. 

By compactness, we can modify $\hat f$ so that it is strictly increasing. In fact, in what follows we will view $\hat f$ as a definable, increasing, partial function  $\hat f\subset D\times D$. Since $\hat f$ is increasing and $\hat f(p(M))=p(M)$, $\hat f(C)\subseteq C$.

Let $\phi(x,y)$ be a formula implying $x<y$ and expressing  ``$[x,y]\subset \dom (\hat f)$ and $\hat f\restriction[x,y]$ is an order-isomorphism of $(D,<)$-intervals $[x,y]$ and $[\hat f(x),\hat f(y)]$''. Then:
$$ p(x)\cup p(y)\cup\{x<y\}\vdash \phi(x,y).$$
By compactness, there is a formula $\theta(x)\in p(x)$   such that:
\begin{equation}
\models  \forall x\forall y  (\theta(x)\land \theta(y)\land x<y\rightarrow \phi(x,y)).
\end{equation}
Since $p$ is a simple type and $\theta(x)\in p$, the set $\lnot\theta(C)$ is finite, so there is a $n_0\in \omega$ such that $\{c_i\mid i\geq n_0\}\subset \theta(C)$. For all $n\geq n_0$, by (1),  $f$ maps the (two-element) interval $[c_n,c_{n+1}]$ onto the interval $[\hat f(c_n),\hat f(c_{n+1})]$, so  
$\hat f(c_n)$ and $\hat f(c_{n+1})$ are consecutive elements of $D$, i.e $\hat f(S(c_n))=S(\hat f(c_n))$ (where $S$ is the succesor function on $(D,<)$). By induction we easily get  $\hat f(S^m(c_{n_0}))=S^m(\hat f(c_{n_0}))$ for all $m\in \omega$. Since $\hat f(c_{n_0})\in C$, there is a $k\in \mathbb Z$ with $\hat f(c_{n_0})=S^k(c_{n_0})$. Then for all $n\geq n_0$: 
$$\hat f(c_n)=\hat f(S^{n-n_0}(c_{n_0}))=    S^{n-n_0}(\hat f(c_{n_0}))=S^{n-n_0+k}(c_{n_0})=S^k(S^{n-n_0}(c_{n_0}))=S^k(c_n).$$ 
Hence  $(\hat f(x)=S^k(x))\in \tp(c_n)$ for  all $n\geq n_0$.  Since $p$ is a $C$-type, $(\hat f(x)=S^k(x))\in p(x)$.
\end{proof}
 
\begin{Lemma}\label{Lemma_x<cl(y)}
Suppose that the type $p\in S_1(T)$ is simple, as witnessed by $(C,<)$  and $(D,<)$. Let $\cl(x)=\{S^n(x)\mid n\in\mathbb Z\}$. Then   $\Pi(x,y)=p(x)\cup p(y)\cup x<\cl(y)$ has a unique completion in $S_2(T)$. 
\end{Lemma}
\begin{proof}First we {\it claim} that $\tp(a/b)$ is a $C$-type for all $(a,b)\models \Pi$. Suppose not. That is, there are $a,b\in p(M)$ with  $a<\cl(b)$, and there is a formula   $\phi(x,b)\in \tp(a/b)$ with $\phi(M,b)\cap C=\emptyset$.  Since $a\in \phi(M,b)$, we may apply Lemma \ref{Lemma_basic_simple}(c): let $a'=\min \phi(M,b)$. Then $a'\in p(M)$ and  $a'\leq a<\cl(b)$. Hence $y=\min\phi(M,x)$ relatively defines a function  $f:p(M)\to p(M)$ satisfying $f(b)=a'$. By Lemma 
\ref{Lemma_f_is_Sn},  for some integer $k$ we have $a'=S^k(b)\in\cl(b)$, contradicting $a'<\cl(b)$. Therefore, $\tp(a/b)$ is a $C$-type. 

Since $M$ is $\aleph_0$-saturated, to prove that $\Pi$ has a unique completion, it suffices to show that $\tp(a,b_1)=\tp(a,b_2)$ holds for any pair $(a,b_1),(a,b_2)$  of realizations of $\Pi$.  Without loss of generality, we will  assume $\models a<b_1\leq b_2$ and  prove that $\psi(x,b_2)\in \tp(a/b_2)$ implies $\psi(x,b_1)\in\tp(a/b_1)$. Observe that, by the above claim, both $\tp(a/b_1)$ and $\tp(a/b_2)$ are $C$-types.
Also observe that for all $m\in\mathbb Z$ the element   $S^m(b_2)$ realizes $p$, so $\models \psi(c,b_2)\leftrightarrow \psi(c,S^m(b_2))$ holds for all $c\in C$. Since $\tp(a/b_2)$ is a $C$-type  $\models \psi(a,b_2)\leftrightarrow \psi(a,S^m(b_2))$ holds, so $\psi(x,b_2)\in\tp(a/b_2)$ implies $S^m(b_2)\in\psi(a,M)$ and $\cl(b_2)\subseteq \psi(a,M)$. 
Define: 
$$D_0:=\bigcup\{[x,b_2]\mid a\leq x \mbox{ and } [x,b_2]\subseteq \psi(a,M)\}.$$ 
Clearly, $D_0$ is a convex, $ab_2$-definable subset of $\psi(a,M)$. By Lemma \ref{Lemma_basic_simple}(c), there exists $b=\min D_0$. If  $b=a$,  then we are done: $b_1\in [a,b_2]\subseteq\psi(a,M)$ implies $\psi(x,b_1)\in\tp(a/b_1)$. So suppose $a<b$. 
Then $S^{-1}(b)\notin D_0$ implies  $\models  \lnot\psi(a,S^{-1}(b))\land \psi(a,b)$. Since $\tp(S^{-1}(b))=\tp(b)$, the formula $\lnot\psi(x,S^{-1}(b))\land \psi(x,b)$ is not satisfied by an element of $C$, so $\tp(a/b)$ is not a $C$-type. 
Then  the {\it claim} implies $a\in \cl(b)$ which combined with $a<\cl(b_1)$ implies $b<b_1\leq b_2$. Since $D_0$ is convex and contains $b,b_2$, we conclude $b_1\in D_0$ and, in particular, $\psi(x,b_1)\in \tp(a/b_1)$, as desired. 
\end{proof}

\begin{Lemma}\label{Lemma_cl_n}
Suppose that $p\in S_1(T)$ is a simple type, as witnessed by $(C,<)$ and $(D,<)$.
For all finite nonempty sets $A\subset p(M)$ the partial type $\Pi_A(x)=p(x)\cup x<\cl(\min A)$ has a unique completion $p_A(x)\in S_1(A)$.  The type $p_A(x)$ is simple, as witnessed by $(C,<)$ and $(D,<)$.
\end{Lemma}
\begin{proof}First we show that the second claim follows from the first one: Note that every $C$-type from $S_1(A)$ contains the type $\Pi_A(x)$, because every formula from $\Pi_A(x)$ is satisfied by all but finitely many elements of $C$. Therefore, if $p_A(x)$ is the unique completion of $\Pi_A(x)$, then it is also the unique $C$-type in $S_1(A)$;  $p_A$ is a simple type, as witnessed by $(C,<)$.

The first claim we prove by induction on $|A|$.  Lemma \ref{Lemma_x<cl(y)} proves the case $|A|=1$. 
Suppose that the lemma holds for all simple types (in any theory) and all subsets of $p(M)$  of size at most $n$. Let $A=\{a_1,...a_{n+1}\}$ where $a_{n+1}\leq a_n\leq ...\leq a_{1}$. Then $\Pi_A(x)=p(x)\cup\{x<a_{n+1}\}$. 

Case 1. \ $a_{n+1}\in \cl(a_n)$. \ In this case $\cl(a_n)=\cl(a_{n+1})$, so the types $\Pi_{\{a_1...a_n\}}(x)=p(x)\cup x<\cl(a_n)$ and $\Pi_{A}(x)=p(x)\cup x<\cl(a_{n+1})$ 
are equal and, by the induction hypothesis, the type $\Pi_A(x)$ has a unique completion $p_{\{a_1...a_n\}}\in S_1(a_1,...,a_n)$. Since 
$a_{n+1}\in \dcl(a_n)$, the type $p_{\{a_1...a_n\}}$ has a unique extension $p_A\in S_1(A)$; $p_A$  is a unique extension of $\Pi_A(x)$, as well.   

Case 2. \ $a_{n+1}<\cl(a_n)$. \ Let $b,b'\models \Pi_A$ and we will prove $\tp(b/A)=\tp(b'/A)$. By the induction hypothesis, the type $p(x)\cup x<\cl(a_n)$ has a unique extension $p_{\{a_1,...,a_n\}}(x)\in S_1(a_1,...,a_n)$. By the first paragraph of the proof the type  $p_{\{a_1,...,a_n\}} $ is  simple, as  witnessed by $(C,<)$. Note that $b,b',a_{n+1}$ realize $\Pi_{\{a_1,...,a_n\}}(x)$, so by the induction hypothesis, they realize $p_{\{a_1.,..,a_n\}}$, too. By Lemma \ref{Lemma_x<cl(y)}  the type $p_{\{a_1,...,a_n\}}(x)\cup p_{\{a_1,...,a_n\}}(y) \cup x<\cl(y)$ has a unique completion   $q\in S_2(a_1,...,a_n)$. Note that $(b,a_{n+1}), (b',a_{n+1})$ realize $q$, so  $\tp(b/A)=\tp(b'/A)$. Therefore, the type $\tp(b/A)$ is a  unique completion of $\Pi_A(x)$, as desired.
\end{proof}
As an immediate corollary we have the following.
\begin{Corollary}\label{Cor_pntypes}
Suppose that $p\in S_1(T)$ is a simple type, as witnessed by $(C,<)$. Then for all $n\geq 1$ the type 
 \ $\Pi^n(\bar x):=\bigcup_{i=0}^n p(x_i)\cup \bigcup_{i=0}^{n-1} x_{i+1}<\cl(x_i)$ \
has a unique completion in $S_{n+1}(T)$. 
\end{Corollary}

\begin{proof}[{\it Proof of Theorem 2}]  Fix a countable, discrete, endless order $\mathbb L=(L,<_L)$ and, by saturation,  find its isomorphic copy $A_L=\{ a_i\mid i\in L\}\subset p(M)$  which is closed under the successor function of $(p(M),<)$. 
Consider the type $\Sigma(x)=p(x)\cup \{x\neq a_i\mid i\in L\}$. We {\it claim} that there is a countable model $M_{\mathbb L}$ omitting $\Sigma$. Suppose not. Then, by the omitting types theorem, there exists a consistent formula $\phi(x, \bar a)$ ($\bar a\subset A_L$) forcing $\Sigma(x)$. Choose $\phi$ with $|\bar a|=n+1$ minimal possible and, without loss of generality, let $\bar a=(a_0,...,a_n)$. 
The minimality assumption guarantees that the sets $\cl(a_i) \ ( i\leq n)$ are pairwise disjoint. Since the formula $\phi(x,\bar a)$ is consistent with $p(x)$, the set $\phi(M,\bar a)$ meets $p(M)$ and Lemma \ref{Lemma_basic_simple}(c) applies:  $b=\min \phi(M,\bar a)$ is well-defined. Since $\phi(x,\bar a)\vdash \Sigma(x)$, $b\in p(M)$ and $\cl(b)\cap \cl(a_i)=\emptyset$ for all $i\leq n$. 
Hence the sets $\cl(b), \cl(a_0),...,\cl(a_{n})$ are pairwise disjoint. Re-enumerate the tuple $a_0,...,a_n,b$ in increasing (in $p(M)$) order, and note that the new tuple satisfies the type $\Pi^{n+1}(\bar x)$ from Corollary \ref{Cor_pntypes}.  
Since $\cl(S^{-1}(b))=\cl(b)$,  the same re-enumeration applied to the tuple $a_0,...,a_n, S^{-1}(b)$ satisfies $\Pi^{n+1}(\bar x)$, too. By Corollary \ref{Cor_pntypes} we have $\tp(\bar a,b)=\tp(\bar a,S^{-1}(b))$ and, in particular, $\models  \phi(S^{-1}(b),\bar a)$. That contradicts $b=\min\phi(M,\bar a)$, proving the claim and Theorem 2 as $p(M_{\mathbb L})\cong \mathbb L$. 
\end{proof}

\section{Corollaries and further questions}
\begin{Lemma}\label{Lemma_reldef_DO}
Suppose that $p\in S_1(T)$, $<_p$ is a relatively definable discrete order on $p(M)$ and $S_p$ is the successor function on $(p(M),<_p)$. Then for all $a\in p(M)$ there exists  an infinite $\{a\}$-definable discrete  order whose initial part is $\{S_p^n(a)\mid n\in\omega\}$. 
\end{Lemma}
\begin{proof}
Let $(D,<)$ be a definable linear order extending $(p(M),<_p)$. 
 First we prove that $S_p$ is a relatively definable function on $p(M)$. Let $a\in p(M)$ and $b=S_p(a)$.  Then the set $p(x)\cup\{a<x<b\}$ is inconsistent. By compactness, there is a formula $\theta(x)\in p$, without loss  implying $x\in D$,  such that $\theta(x)\vdash \lnot (a<x<b)$. Hence $b$ is the immediate successor of $a$ in the order $(\theta(M),<)$; $b=S_{\theta}(a)$, where the successor function   $S_{\theta}$ is definable. It is easy to see that $y=S_{\theta}(x)$ relatively defines $S_p$.

Let $a\in p(M)$ and let $D_0$ be the union of all intervals $[a,x]$ of $(\theta(M),<)$ that are discretely ordered by $<$.
Then $(D_0,<)$ satisfies the conclusion of the lemma.
\end{proof}

\begin{Corollary}
If $T$ admits an infinite, relatively definable, discrete order on the locus of some complete type, then $I(\aleph_0,T)=2^{\aleph_0}$. 
\end{Corollary}

\begin{Corollary}\label{Corolary_Econvex}
Suppose that $I(\aleph_0,T)<2^{\aleph_0}$,  $<$ is a relatively definable linear order on the locus of $p\in S_1(T)$ and $E$ is a relatively definable convex\footnote{Classes are convex.} equivalence relation with infinitely many classes on $p(M)$. Then the quotient order $(p(M)/E,<)$ is dense and endless. 
\end{Corollary}
\begin{proof}
 As in Fact \ref{Fact1},  one shows that there is a definable linear order  $(D,<)$  extending $(p(M),<)$ and a definable   convex (on $(D,<)$) equivalence relation $E_D$ extending $E$. Now, switch to the $E_D$-sort in $T^{eq}$ (whose domain is $D/E_D$) and consider the type $p_{/E_D}$ there; it is linearly ordered by $<$ and $(p_{/E_D}(M^{eq}),<_{E_D}) \cong (p(M)/E,<)$. By saturation, the order $(p_{/E_D}(M^{eq}),<_{E_D})$ is  endless and either dense or discrete. The second option is ruled out by $I(\aleph_0,T)<2^{\aleph_0}$. 
\end{proof}

Let  $T$ be an $L_{\omega_1,\omega}$-sentence and let $Mod(T)$ be the set of all models of $T$ whose domain is $\mathbb N$. $Mod(T)$ is endowed with the formula topology, 
where basic clopen sets are $U_{\phi,\bar n}=\{N\in Mod(T)\mid N\models \phi(\bar n)\}$ 
($\phi(\bar x)$ a formula and $\bar n\in\mathbb N^{|\bar x|}$). 
$Mod(T)$ is a Polish space and $\cong$ is an analytic equivalence relation on $Mod(T)$. 
Borel reductions of pairs $(X,E)$ where $X$ is a Borel space and $E$ an equivalence relation on $X$ 
were introduced by Friedman and Stanley in \cite{FS}. 
In our context, pairs are $(Mod(T),\cong)$. $T_1$ is Borel reducible to $T_2$, denoted by $T_1\leq_B T_2$,  if there is a Borel function 
$f:Mod(T_1)\to Mod(T_2)$ such that for all $M_1,M_2\in Mod(T_1)$: $M_1\cong M_2$ if and only if $f(M_1)\cong f(M_2)$. 
$\leq_B$ is a  quasi-order and $T_1\leq_B T_2$ suggests that the  isomorphism problem for $Mod(T_1)$ is not more  complex than that for $T_2$.  Friedman and Stanley defined: 
$T$ is {\it Borel complete} if $T'\leq_B T$ for all $T'$. The intuition behind this definition is that  the isomorphism problem of a Borel complete theory is as complicated as possible.  In \cite{FS} 
several Borel complete theories were found, and among them is the theory of linear orders.  
That is why Theorem 2 naturally motivates asking Question 1. 

\smallskip
Let us re-inspect the proof of Theorem 2. We started with a small theory $T$ interpreting an infinite discrete order.
Then, in order to obtain a simple type, we named a parameter $c$. That did not affect the aimed conclusion $I(\aleph_0,T)=2^{\aleph_0}$, because $I(\aleph_0,T(c))=2^{\aleph_0}$ implies $I(\aleph_0,T)=2^{\aleph_0}$. But, whether that affects Borel completeness of $T$, we do not know:

\begin{Question}
If $T$ is small and $T(\bar a)$ is Borel complete for some  $\bar a\in M^n$, must $T$ be Borel complete?
\end{Question}

In 2015, Chris Laskowski asked a more general question\footnote{Thanks to the referee for pointing this to us.}:

\begin{Question}
Can Borel completeness be gained or lost by naming a constant?
\end{Question}

This question was posed and discussed in Laskowski's conference talk \cite{Chris}, where a partial result due to Richard Rast was mentioned:  $\cong_T$ and $\cong_{T(c)}$ are simultaneously Borel. However, both Question 2 and 3 seem to be widely open.

Going back to the proof of Theorem 2, we have ``coded'' an arbitrary endless, discrete, linear order $\mathbb L$ by a model $M_{\mathbb L}$ with $\mathbb L\cong (p(M_{\mathbb L}),<)$. That clearly implies $I(\aleph_0,T)=2^{\aleph_0}$, but we do not know whether $T$ has to be Borel complete.

It is quite natural to expect that if the theory has  some part that is maximally complicated in some sense, then the theory itself is maximally complicated, too, 
But we could not answer the following question, even assuming that $T$ is small:

 \begin{Question}
If $T_{\theta}$  is Borel complete for some formula $\theta(x)$, must $T$ be Borel complete?
\end{Question}

\end{document}